\documentclass[12pt]{article}

\usepackage{upgreek}

\usepackage{hyperref}

\usepackage[normalem]{ulem}
\usepackage{amsthm}
\usepackage{amsmath}
\usepackage{amssymb}
\usepackage{mathrsfs}
\usepackage{mathtools}
\usepackage{graphicx}
\usepackage{tikz}
\usetikzlibrary{arrows}
\usepackage{enumerate}
\usepackage{comment}

\usepackage{fullpage}

\usepackage{appendix}

\newtheorem{theorem}{Theorem}
\newtheorem{conjecture}[theorem]{Conjecture}
\newtheorem{lemma}[theorem]{Lemma}

\theoremstyle{definition}
 
\theoremstyle{remark}

\theoremstyle{remark}

\numberwithin{theorem}{section}

\newcommand{\E}{\textsf{\upshape E}}
\newcommand{\prob}{\textsf{\upshape Pr}}

\newcommand{\Gnp}{\mathcal{G}(n,p)}

\begin{document}

\title{Linear Colouring of Binomial Random Graphs}

\author{
Austin Eide\thanks{Department of Mathematics, Toronto Metropolitan University, Toronto, ON, Canada; e-mail: \texttt{austin.eide@torontomu.ca}}, 
Pawe\l{} Pra\l{}at\thanks{Department of Mathematics, Toronto Metropolitan University, Toronto, ON, Canada; e-mail: \texttt{pralat@torontomu.ca}}
}

\maketitle

\begin{abstract}
We investigate the linear chromatic number $\chi_{\text{lin}}(\Gnp)$ of the binomial random graph $\Gnp$ on $n$ vertices in which each edge appears independently with probability $p=p(n)$. For dense random graphs ($np \to \infty$ as $n \to \infty$), we show that asymptotically almost surely $\chi_{\text{lin}}(\Gnp) \ge n (1 - O( (np)^{-1/2} ) ) = n(1-o(1))$. Understanding the order of the linear chromatic number for subcritical random graphs ($np < 1$) and critical ones ($np=1$) is relatively easy. However,  supercritical sparse random graphs ($np = c$ for some constant $c > 1$) remain to be investigated. 
\end{abstract}

\section{Introduction} 

Let $G = (V,E)$ be a graph and let $\phi: V \to \{1,\dots,k\}$ be an assignment of $k$ colours to the vertices of $G$. We say that $\phi$ is a \textit{proper} $k$-\textit{colouring} if for each $\{v,w\} \in E$, $\phi(v) \neq \phi(w)$. The \textit{chromatic number} of $G$, denoted $\chi(G)$, is the smallest positive integer $k$ such that a proper $k$-colouring of $G$ exists.

\medskip

Given a colouring $\phi$ and subset $S \subseteq V$, we say that a vertex $v \in S$ is a \textit{centre} for $S$ if $\phi(v)$ is distinct from $\phi(w)$ for all $w \neq v$ in $S$. A \textit{centred} $k$-\textit{colouring} of $G$ is a $k$-colouring of $G$ such that for every connected subgraph $H \subseteq G$, $V(H)$ has a centre. The \textit{centred chromatic number} $\chi_{\text{cen}}(G)$ is the smallest $k$ such that a centred $k$-colouring of $G$ exists. Observe that a centred colouring is necessarily proper, since each edge $\{v,w\} \in E$ comprises a connected subgraph of $G$. Hence we have the inequality $\chi(G) \leq \chi_{\text{cen}}(G)$.

The centred chromatic number is an important and natural graph parameter that has been introduced under numerous names in the literature: rank function~\cite{nevsetvril2003order}, vertex ranking number (or ordered colouring)~\cite{deogun1994vertex}, weak colouring number~\cite{kierstead2003orderings}. Its study was systematically undertaken by Ne\v{s}et\v{r}il and Ossona de Mendez under the name of \emph{tree-depth}~\cite{nevsetvril2006tree}. The notion of tree-depth is related to the one of tree-width. The tree-width of a graph can be seen as a measure of closeness to a tree, while the tree-depth takes also into account the diameter of the tree. Both serve as important measures of sparsity of a graph~\cite{nevsetvril2012sparsity,nevsetvril2015low}.

\medskip

A \textit{linear} $k$-\textit{colouring} of $G$ is a $k$-colouring such that every subgraph of $G$ that is a path has a centre. The corresponding \textit{linear chromatic number} $\chi_{\text{lin}}(G)$ is defined in the obvious way. A linear colouring is necessarily proper, as each edge $\{v,w\} \in E$ is a path of length one. On the other hand, a centred colouring is necessarily linear, since path subgraphs are connected. Therefore, we have $\chi(G) \leq \chi_{\text{lin}}(G) \leq \chi_{\text{cen}}(G)$. (In other works, e.g., \cite{yuster1998linear}, the term \textit{linear colouring} has been used to refer to proper colourings with the property that the subgraph induced any pair of colour classes is a disjoint union of paths. This is distinct from the meaning here.)

The linear chromatic number was introduced by Kun, O'Brien, Pilipczuk, and Sulivan~\cite{kun2021polynomial} who were motivated by finding efficiently-computable approximations of tree-depth in the class of bounded expansion graphs. The authors of~\cite{kun2021polynomial} provide a family of graphs that contains, for every $\epsilon > 0$, a graph $G$ with $\chi_{\text{cen}}(G) > (2-\epsilon) \chi_{\text{lin}}(G)$ and based on that they stated the following, quite bold, conjecture:
\begin{conjecture}[\cite{kun2021polynomial}]\label{conjecture}
For all graphs $G$, $\chi_{\text{cen}}(G) \leq 2\chi_{\text{lin}}(G)$. 
\end{conjecture}
We are far away from proving this conjecture. The only class of graphs for which centred chromatic number is known to be bounded by a linear function of linear chromatic number is the class of bounded degree trees~\cite[Theorem~4]{kun2021polynomial}. Currently the best upper bound is proved by Bose, Dujmovi{\'c}, Houdrouge, Javarsineh, and Morin~\cite{bose2022linear} who were able to prove that 
$$
\chi_{\text{cen}}(G) \leq \chi_{\text{lin}}(G)^{10} \Big( \log (\chi_{\text{lin}}(G)) \Big)^{O(1)}.
$$
This result improved the bound proved by Czerwi\'nski, Nadara, and Pilipczuk~\cite{czerwinski2021improved} (they reduced the exponent from 190 to 19) which, in turn, improved the original bound by Kun at al.~\cite{kun2021polynomial} (with exponent 190). Bose at al.~\cite{bose2022linear} provide further evidence in support of the conjecture by establishing that, if $G$ is a $k \times k$ pseudo-grid, then $\chi_{\text{cen}}(G) = O (\chi_{\text{lin}}(G))$. 

\medskip

In this paper, we investigate the binomial random graph $\Gnp$ that is formally defined as a distribution over the class of graphs with the set of nodes $[n]:=\{1,\ldots,n\}$ in which every pair $\{i,j\} \in \binom{[n]}{2}$ appears independently as an edge in $G$ with probability~$p$. Note that $p=p(n)$ may (and usually does) tend to zero as $n$ tends to infinity. Most results in this area are asymptotic by nature. We say that $\Gnp$ has some property \emph{asymptotically almost surely} (or \emph{a.a.s.}) if the probability that $\Gnp$ has this property tends to $1$ as $n$ goes to infinity. For more about this model see, for example,~\cite{Bollobas,JLR,frieze2016introduction}.

\medskip

The binomial random graph $\Gnp$ is notoriously a good candidate for constructing counterexamples to conjectures that seem to be false, including the seminal result of Erd\H{o}s from 1959~\cite{erdos1959graph} that ``many consider [to be] one of the most pleasing uses of the probabilistic method, as the result is surprising and does not appear to call for nonconstructive techniques'' (see~\cite{alon2016probabilistic}). The \emph{girth} of a graph is the size of its shortest cycle. Erd\H{o}s showed in~\cite{erdos1959graph} that for any $k$ and $\ell$ there exists a graph $G$ with girth more than $\ell$ and $\chi(G) > k$. 

Alternatively, one can investigate random graphs to support various conjectures that seem to be true. In particular, $\Gnp$ with $p=1/2$ yields a uniform distribution of (labeled) graphs on $n$ vertices, so showing that a given conjecture holds a.a.s.\ for $\mathcal{G}(n,1/2)$ is equivalent to proving that almost all graphs satisfy the conjecture. Many open problems are supported by such statements including the following, clearly biased, small sample of results of this flavour: Meyniel's conjecture~\cite{pralat2016meyniel,pralat2019meyniel}, Tutte's conjecture~\cite{pralat2020almost}, and Jaeger's conjecture~\cite{delcourt20239regular}. 

\medskip

The results presented in this paper for dense binomial random graphs $\Gnp$ (that is, in the regime when $np \to \infty$) support Conjecture~\ref{conjecture}. Our main theorem is the following. 

\begin{theorem}\label{main_theorem} 
Let $\omega = \omega(n) \le n$ be any function that tends to infinity as $n \to \infty$, and let $p = \omega / n$. Then, the following holds a.a.s.: 
$$
\chi_{\text{lin}}(\Gnp) \geq n - \frac{510n}{\sqrt{\omega}}.
$$
\end{theorem}

In our proofs, we did not try to optimize the constants. Since $\chi_{\text{lin}}(G) \leq \chi_{\text{cen}}(G)$ and, trivially, $\chi_{\text{cen}}(G) \le n$, Theorem~\ref{main_theorem} implies that a.a.s.\ $\chi_{\text{lin}}(G(n,p)) = (1+o(1)) \chi_{\text{cen}}(G(n,p)) = (1+o(1)) n$. In particular, we conclude that Conjecture~\ref{conjecture} holds for almost all graphs. 

\medskip

Supporting Conjecture~\ref{conjecture} is a nice implication but understanding the behaviour of the linear chromatic number for $\Gnp$ seems to be interesting on its own. In particular, our result implies the lower bound for the centred chromatic number of dense binomial random graphs proved in~\cite{perarnau2014tree}, where it was shown that $\chi_{\text{cen}}(G(n, \omega/n)) \geq n - O\left(n / \sqrt{\omega} \right)$ a.a.s. 

Investigating the linear chromatic number for very sparse random graphs, before the giant component is formed, is relatively easy. Observations from~\cite{perarnau2014tree} give us the following.  

\begin{theorem}\label{sparse_linear}
    The following holds a.a.s.: 
    $$
    \chi_{\text{lin}}(\mathcal{G}(n,c/n)) = 
    \begin{cases} 
    \Theta(\log \log n) \quad & \text{if} \quad c \in (0,1) \\ 
    \Theta(\log n) \quad & \text{if} \quad c = 1. 
    \end{cases}
    $$
\end{theorem}

On the other hand, supercritical sparse random graphs (when $p=c/n$ for some constant $c > 1$) remain to be investigated. The proof of our main result, Theorem~\ref{main_theorem}, can be adjusted to show that a.a.s.\ $\chi_{\text{lin}}(\mathcal{G}(n,c/n)) = \Theta(n)$, provided that $c$ is large enough. However, there seems to be no hope to apply the current argument to prove it for any $c > 1$. Maybe a.a.s.\ $\chi_{\text{lin}}(\mathcal{G}(n,c/n)) = o(n)$ for some $c > 1$? That would show that Conjecture~\ref{conjecture} is false, since a.a.s.\ $\chi_{\text{cen}}(\mathcal{G}(n,c/n)) = \Theta(n)$ for any $c > 1$. 

\medskip

The paper is structured as follows. We first provide a high level sketch of the proof of the main theorem, Theorem~\ref{main_theorem} (see Subsection~\ref{sec:sketch}). Section~\ref{sec:proof_main} is devoted to the proof of Theorem~\ref{main_theorem}. Observations that prove Theorem~\ref{sparse_linear} can be found in Section~\ref{sec:proof_sparse}.

\subsection{Sketch of the Proof of Theorem~\ref{main_theorem}}\label{sec:sketch}

The starting point of the proof is an idea from the paper of Alon, McDiarmid, and Reed~\cite{alon1991acyclic} on acyclic colourings. Let $x = x(n) \in (0,1)$ and consider any colouring of the vertices of $\Gnp$ which uses as most $(1-x)n$ colour classes. By removing at most one vertex from each class, we can make the sizes of all classes even. Since we remove at most $(1-x)n$ vertices, a set $S$ of size at least $xn$ remains. Vertices in $S$ are necessarily in even classes and so of size at least $2$. In particular, colours that were initially present only one time disappeared. Finally, we (arbitrarily) pair the vertices within each colour class, resulting in at least $xn / 2$ pairs of vertices, where each pair is a subset of a single colour class. Let $\mathcal{P}$ be the set of pairs formed at this step. 

We call a path in $\Gnp$ \textit{bad} if it has no centre, and observe that any path on vertices from $S$ which visits each pair of $\mathcal{P}$ either exactly twice or not at all is bad. To show that the coloring we started with is \textit{not} linear, we seek a bad path for the pairing $\mathcal{P}$. Maybe in each pairing $\mathcal{P}$ there is always a short bad path? The answer is `no'---it is relatively easy to construct a large set of pairs with no short bad paths a.a.s. Alternatively, one might simply look for a Hamilton path on the vertices in $S$. This also turns out to be too much to ask for as, in general, the subgraph of $\Gnp$ induced by the vertices in $S$ may not even be connected. Indeed, there are many isolated vertices in $\Gnp$ for $p$ below the threshold for connectivity $\bar{p} = \log n / n$ so this subgraph can have many isolated vertices.

However, something slightly weaker turns out to be true. By repeatedly removing pairs of vertices in $S$ which contain a vertex of small degree until no such pairs remain, we reach a subset $S' \subseteq S$ and a sub-pairing $\mathcal{P}' \subseteq \mathcal{P}$. (This procedure is reminiscent of the construction of the $k$-\textit{core} of the subgraph induced by $S$.) Provided that $x$ is large enough, one can show that, a.a.s., not too many pairs are removed and that the resulting set $S'$ induces a connected subgraph with good expansion. Using the now-standard rotation-extension technique of P\'{o}sa~\cite{posa1976hamiltonian}, it can then be shown that this subgraph has a Hamilton path a.a.s. Since $\mathcal{P}' \subseteq \mathcal{P}$, any such path is bad, and hence the colouring we started with is a.a.s.\ not linear.

\section{Dense Case: $np \to \infty$ (Proof of Theorem~\ref{main_theorem})} \label{sec:proof_main}

We will use the following specific instances of Chernoff's bound. Let $X \in \textrm{Bin}(n,p)$ be a random variable distributed according to a Binomial distribution with parameters $n$ and $p$. Then, a consequence of \emph{Chernoff's bound} (see e.g.~\cite[Theorem~2.1]{JLR}) is that for any $t \ge 0$ we have
	\begin{eqnarray}
		\prob( X \ge \E X + t ) &\le& \exp \left( - \frac {t^2}{2 (\E X + t/3)} \right)  \label{chern1} \\
		\prob( X \le \E X - t ) &\le& \exp \left( - \frac {t^2}{2 \E X} \right).\label{chern}
	\end{eqnarray}

We define a \textit{set-pairing} to be a pair $(S, \mathcal{P})$ where $S \subseteq [n]$ is a set of even size and $\mathcal{P}$ is a set of the form 
$$
\{\{v_{1},v_{2}\},\{v_{3},v_{4}\}, \dots, \{v_{|S|-1}, v_{|S|}\}\}
$$
where $v_{1},v_{2},\dots,v_{|S|}$ is some ordering of the vertices of $S$. Given set-pairings $(S,\mathcal{P})$ and $(S',\mathcal{P}')$, we say that $(S',\mathcal{P}') \subseteq (S, \mathcal{P})$ if and only if $S' \subseteq S$ and $\mathcal{P}' \subseteq P$.

\medskip

Throughout this section, $\omega = \omega(n)$ will denote a function of $n$ which grows to infinity arbitrarily slowly and satisfies $\omega \leq n$ so that $p=\omega/n \le 1$. Recall that for a graph $G$ on vertex set $[n]$ and $S \subseteq [n]$, we let $G[S]$ denote the subgraph of $G$ induced by the vertices in~$S.$

\medskip

Let $(S, \mathcal{P})$ be a set-pairing and let $G$ be a graph on vertex set $[n]$. For a given $k \geq 0$ we define the $k$-\textit{core of} $(S,\mathcal{P})$ in $G$, denoted $C_{k}^{G}(S, \mathcal{P})$ to be the maximal induced subgraph of $G[S]$ with minimum degree at least $k$ and such that if $v \in V(C_{k}^{G}(S, \mathcal{P}))$ and $\{v,w\} \in \mathcal{P}$, then $w \in V(C_{k}^{G}(S,\mathcal{P}))$. Note that any maximal subgraph satisfying these conditions is necessarily unique, else a larger subgraph satisfying the same conditions could be constructed by taking a union. Thus the definition is unambiguous (though the $k$-core may be empty). Moreover, to find the $k$-core of $(S,\mathcal{P})$ one may repeatedly remove vertices of degree less than $k$ (together with their partners in $\mathcal{P}$) until there is no vertex of degree less than $k$. 

The key feature of $C_{k}^{G}(S, \mathcal{P})$ is that it respects the original pairing $\mathcal{P}$: for any $\{v,w\} \in \mathcal{P}$, either both $v$ and $w$ are in the $k$-core, or neither of them is. In this subsection, we establish some properties of $k$-cores of set pairings $(S,\mathcal{P})$ in the binomial random graph $\Gnp$. When the host graph is clear from context, we simply write $C_{k}(S, \mathcal{P})$. 

\medskip

We will first show that $k$-cores are large (Subsection~\ref{cores}) and have good expansion properties (Subsection~\ref{expansion}). The results in these two subsections are adaptations of similar results in~\cite[Section 3]{krivelevich2014cores} to the present application. These observations, combined via sprinkling with the rotation-extension technique of P\'{o}sa, imply that the corresponding $k$-cores have Hamilton paths (Subsection~\ref{long_paths}).

\subsection{$k$-cores are Large}\label{cores}

Our first lemma shows that, a.a.s., for every set-pairing $(S, \mathcal{P})$ with $|S| \geq cn / \sqrt{\omega}$, the core $C_{|S|p/3}(S, \mathcal{P})$ in $\Gnp$ has at least $|S|/2$ vertices.

\begin{lemma}\label{k_core}
Let $p = \omega/n$, $c > 0$, and let $G = \Gnp$. Then, a.a.s.\ for every set-pairing $(S, \mathcal{P})$ with $|S| \geq c n / \sqrt{\omega}$, there exists a set-pairing $(S', \mathcal{P}') \subseteq (S, \mathcal{P})$ such that $|S'| \geq |S|/2$ and the subgraph $G[S']$ has minimum degree at least $|S|p / 3$.
\end{lemma}

\begin{proof} 
First, fix a set-pairing $(S, \mathcal{P})$ with $|S| \geq cn / \sqrt{\omega} = \Omega(\sqrt{n})$. We build a sub set-pairing $(S', \mathcal{P}')$ using the following simple algorithm. Set $S_{0} = S$ and $\mathcal{P}_{0} = \mathcal{P}$. For $i = 0, 1, 2,\dots$, if $G[S_{i}]$ contains a vertex of degree less than $|S|p/3$, then let $v_{i+1}$ be the smallest such vertex, and let $w_{i+1}$ be its partner such that $\{v_{i+1},w_{i+1}\} \in \mathcal{P}_{i}$; set $S_{i+1} = S_{i} \setminus \{v_{i+1},w_{i+1}\}$ and $\mathcal{P}_{i+1} = P_{i} \setminus \{\{v_{i+1},w_{i+1}\}\}$. If no vertex of degree less than $|S|p/3$ exists in $G[S_{i}]$ (which is trivially true if $S_{i} = \varnothing$), then the algorithm terminates after $i$ steps. 

Let $T = T(S,\mathcal{P})$ be the termination time of the algorithm and let $S' = S_{T}$. If $T < |S|/2$, then clearly $G[S']$ is a subgraph of $\Gnp$ on $|S| - 2T$ vertices of minimum degree at least $|S|p/3$. To prove the lemma, we will show that a.a.s.\ $T(S,\mathcal{P}) \leq |S|/4$ for all set-pairings $(S, \mathcal{P})$ with $|S| \geq cn/\sqrt{\omega}$. 

For each $i \geq 1$, let $B_{i} = \{v_{1},v_{2},\dots,v_{i}\}$, where the $v_{j}$'s are as defined in the algorithm. By construction, we have $|E(B_{i}, S_{i})| < i \cdot |S|p/3$. Suppose that $T > |S|/4$. Then, at step $t =\left\lfloor |S|/4 \right\rfloor$ of the algorithm, we find disjoint sets $B_{t}$ and $S_{t}$ of sizes $\lfloor |S|/4 \rfloor = (1+o(1)) |S|/4$ and $|S| - 2 \lfloor |S| / 4 \rfloor = (1+o(1)) |S|/2$, respectively, such that 
$$
|E(B_{t}, S_{t})| < \left\lfloor \frac{|S|}{4} \right\rfloor \cdot \frac{|S|p}{3} \le \frac{|S|^{2}\omega}{12n}.
$$ 

The preceding shows that for any set-pairing $(S, \mathcal{P})$ with $|S| \geq cn / \sqrt{\omega}$, the event $\{T(S,\mathcal{P}) > |S| / 4 \}$ implies the existence of a pair of disjoint sets $P, Q \subseteq [n]$ such that $|P| = \left\lfloor |S|/4 \right\rfloor$, $|Q| = |S|- 2 \lfloor |S|/4 \rfloor$, and $|E(P,Q)| < \frac{|S|^{2} \omega}{12n}$. Thus,
\begin{equation}\label{union_bound}
\prob\left(\bigcup_{(S,\mathcal{P})}\left\{T(S,\mathcal{P}) > \frac{|S|}{4}\right\}\right) \leq \prob\left(\bigcup_{P,Q}\left\{|E(P,Q)| < \frac{s^{2} \omega}{12n}\right\}\right),
\end{equation}
where the union on the left is taken over all set-pairings with $|S| \geq cn / \sqrt{\omega}$ and the union on the right is over all disjoint $P,Q$ with $|P| = \left\lfloor s/4 \right\rfloor$, $|Q|=s - 2 \left\lfloor s/4 \right\rfloor$, and $s \geq cn / \sqrt{\omega}$.

Consider a fixed pair of disjoint sets $P,Q$ of sizes $\lfloor s/4 \rfloor$ and $s - 2 \lfloor s/4 \rfloor$, respectively, for some $s \geq cn / \sqrt{\omega}$. The cut size $|E(P,Q)|$ is the binomial random variable $X \sim \textrm{Bin}(|P||Q|,p)$ with mean
$$
\E[X] = |P||Q|p = (1+o(1))\frac{s^{2} \omega}{8n}.
$$
From Chernoff's bound~(\ref{chern}) applied with $t=\E[X]-\frac{s^{2} \omega}{12n}=(1+o(1))\E[X]/3$, we then get 
\begin{eqnarray*}
	\prob\left(|E(P,Q)| < \frac{s^{2} \omega}{12n}\right) &\leq& \exp\left\{- (1+o(1)) \frac{\E[X]}{18} \right\}  ~=~ \exp\left\{-(1+o(1))\frac{s^{2}\omega}{144n}\right\} \\
	&\leq&\exp\left\{-\frac{s^{2}\omega}{150n} \right\} ~\leq~ \exp\left\{-\frac{cs\sqrt{\omega}}{150} \right\},
\end{eqnarray*}
where the second equality holds for $n$ sufficiently large, and in the final equality we use the fact that $s \geq cn / \sqrt{\omega}$. For a given $s$, the number of choices for the sets $P$ and $Q$ is at most
$$
\binom{n}{\left\lfloor s/4 \right\rfloor}\binom{n}{s - 2\left\lfloor s/4 \right\rfloor} \leq n^{O(1)}\left(\frac{4ne}{s}\right)^{s/4}\left( \frac{2ne}{s}\right)^{s/2} = n^{O(1)}\left( \frac{2^{\frac{4}{3}}ne}{s} \right)^{\frac{3s}{4}}.
$$
(The $n^{O(1)}$ factor is the price paid for getting rid of ceilings; the constant implied in the $O(\cdot)$ notation does not depend on $s$.) Using the fact that $s \geq cn/\sqrt{\omega}$, the right-hand side above is at most
$$
n^{O(1)}\left( \frac{2^{\frac{4}{3}}\sqrt{\omega}e}{c} \right)^{\frac{3s}{4}} = \exp\left\{(1+o(1))\frac{3s}{8}\log \omega \right\} \leq \exp\left\{ \frac{s\log \omega}{2} \right\}.
$$
(The inequality holds for $n$ sufficiently large.) Thus, the probability that there exist disjoint sets $P,Q$ of sizes $\lfloor s/4 \rfloor$ and $s - 2 \lfloor s/4 \rfloor$, respectively, such that $|E(P,Q)| \leq \frac{s^{2} \omega}{12n}$ is at most 
$$
\exp\left\{-\frac{cs\sqrt{\omega}}{150} + \frac{s\log\omega}{2} \right\} = \exp\left\{- \frac{cs\sqrt{\omega}}{150} \left( 1 - \frac{75 \log\omega}{c\sqrt{\omega}} \right) \right\} \le \exp\left\{- \frac{cs\sqrt{\omega}}{200} \right\}.
$$
(As always, the inequality holds for $n$ sufficiently large.) It follows that
\begin{eqnarray*}
\prob\left(\bigcup_{P,Q}\left\{|E(P,Q)| \leq \frac{s^{2} \omega}{12n}\right\}\right) &\leq& \sum_{s = \lceil cn/\sqrt{\omega} \rceil}^{n} \exp\left\{- \frac{cs\sqrt{\omega}}{200} \right\} \\
&\leq& n e^{-\Omega(n)} = o(1).
\end{eqnarray*}
Based on~(\ref{union_bound}), we conclude that a.a.s.\ $T(S,\mathcal{P}) \leq |S|/4$ for all set-pairings $(S, \mathcal{P})$ with $|S| \geq cn / \sqrt{\omega}$. This completes the proof of the lemma.
\end{proof}

\subsection{$k$-cores are Good Expanders}\label{expansion}

The previous lemma, Lemma~\ref{k_core}, shows that, a.a.s., for every set-pairing $(S, \mathcal{P})$ with $|S| \geq 2cn / \sqrt{\omega}$, the core $C_{|S|p/3}(S, \mathcal{P})$ in $\Gnp$ has at least $|S|/2 \ge cn / \sqrt{\omega}$ vertices. By definition, the minimum degree of $C_{|S|p/3}(S, \mathcal{P})$ is at least $|S|p/3 \ge |V(C_{|S|p/3}(S, \mathcal{P}))| p / 3$. Our next lemma implies that a.a.s.\ for every set-pairing $(S, \mathcal{P})$ with $|S| \geq 2cn / \sqrt{\omega}$, the core $C_{|S|p/3}(S, \mathcal{P})$ in $\Gnp$ is a good expander.

\begin{lemma}\label{expansion_connectivity}
Let $p = \omega / n$ and $c > 0$. The following properties hold a.a.s.: 
 	\begin{enumerate}
		\item [i)] For any subgraph $H$ of $\Gnp$ on at least $cn/\sqrt{\omega}$ vertices with $\delta(H) \geq |V(H)|p / 3$, we have $|N_{H}(X) \setminus X| > 2|X|$ for every $X \subseteq V(H)$ with $|X| \leq |V(H)| / 45$.
		\item [ii)] Every induced subgraph $H$ of $\Gnp$ with $\delta(H) \geq |V(H)|p/3$ on at least $cn/\sqrt{\omega}$ vertices is connected. 
	\end{enumerate}
\end{lemma}

\begin{proof} 
We begin with i). Suppose that there is a subgraph $H$ of $\Gnp$ on $s \geq cn/\sqrt{\omega}$ vertices with minimum degree at least $sp/3$ that fails the expansion condition in the statement. Let $X \subseteq V(H)$ be a subset of vertices with $|X| \leq s/45$ and such that $|N_{H}(X) \setminus X| \leq 2|X|.$ Then, $N_{H}(X) \setminus X$ is contained in some $Y \subseteq V(H)$, disjoint from $X$, with $|Y| = 2|X|$. In $H$, there are at most $\binom{|X|}{2} + |X||Y| \leq \frac{5}{2}|X|^{2}$ possible edges incident with $X.$ At least 
$$
\frac{\delta(H)|X|}{2} \geq \frac{sp|X|}{6} = \frac{s\omega|X|}{6n}
$$ 
of these edges must be present in $H$, and hence also in $\Gnp$. Writing $|X| = j \le s/45$, the probability that this occurs for a given pair of sets $X$ and $Y$ is at most 
$$
\binom{\left\lfloor \frac{5}{2}j^{2} \right\rfloor}{\left\lceil \frac{s\omega j}{6n} \right\rceil}p^{\lceil s\omega j/6n \rceil} \leq \left( \frac{15ej n}{s\omega} p \right)^{\lceil s\omega j / 6n \rceil} \le \left( \frac{15ej}{s}\right)^{s\omega j / 6n}.
$$

For $s \geq cn/\sqrt{\omega}$, let $\mathcal{B}_{s}$ be the event that there exists a subgraph $H$ of $\Gnp$ with $s$ vertices and minimum degree at least $sp/3$ such that the expansion condition in the statement of the lemma fails. We have 
\begin{eqnarray}
	\prob(\mathcal{B}_{s}) &\leq& \sum_{j=1}^{\lfloor s/45 \rfloor}\binom{n}{j}\binom{n}{2j} \left( \frac{15ej}{s}\right)^{s\omega j/6n} \nonumber \\
    &\leq& \sum_{j=1}^{\lfloor s/45 \rfloor}\left[\frac{1}{4}\left( \frac{ne}{j} \right)^{3}\left(\frac{15ej}{s}		\right)^{s\omega/6n}\right]^{j} \nonumber \\
	&\leq& \sum_{j=1}^{\lfloor s/45 \rfloor}\left[\frac{1}{4}\left(\frac{ne}{j} \right)^{3}\left(\frac{15ej}{s}\right)^{c\sqrt{\omega}/6}\right]^{j}, \label{eq:the_sum_to_bound}
\end{eqnarray}
where in the final inequality we use that $15ej/s < 1$ for $j \leq \lfloor s / 45 \rfloor$ and that $s \geq cn/\sqrt{\omega}$. We will show that the last sum above is $o(1/n)$ uniformly in $s$. This will suffice to finish the proof of part i), since it implies that
$$
\prob\left(\bigcup_{s = \lceil cn/\sqrt{\omega} \rceil}^{n}\mathcal{B}_{s} \right) \leq \sum_{s = \lceil cn/\sqrt{\omega}\rceil}^{n}\prob(\mathcal{B}_{s}) = n\cdot o(1/n) = o(1).
$$

Now, we bound the sum~(\ref{eq:the_sum_to_bound}). We remark first that, since $\omega \leq n$, we have $s \geq cn / \sqrt{\omega} = c\sqrt{n}$ and, in particular, $s \gg \log n$. We will split the sum~(\ref{eq:the_sum_to_bound}) into two parts corresponding to $j \leq \lfloor \log n \rfloor$ and, respectively, $j > \lfloor \log n \rfloor.$

For $1 \leq j \leq \lfloor \log n \rfloor$, we have 

$$
\frac{1}{4}\left( \frac{ne}{j} \right)^{3} \left(\frac{15ej}{s} \right)^{ c\sqrt{\omega}/6 } \leq  \frac{(ne)^{3}}{4} \left(\frac{15e \log n}{c\sqrt{n}} \right)^{c\sqrt{\omega}/6} =: g(n).
$$
where in the inequality we use that $1 \leq j \leq \log n$ and $s \geq  c\sqrt{n}$. It is easy to see that $g(n) = o(1/n)$, and hence
$$
\sum_{j=1}^{\lfloor \log n \rfloor}\left[\frac{1}{4}\left( \frac{ne}{j} \right)^{3}\left(\frac{15ej}{s}\right)^{c\sqrt{\omega}/6}\right]^{j} \leq \sum_{j=1}^{\lfloor \log n \rfloor}(g(n))^{j} = O(g(n)) = o(1/n).
$$

For $\lfloor \log n \rfloor +1 \leq j \leq \left\lfloor\frac{s}{45}\right\rfloor$, observe that
\begin{eqnarray*}
\frac{1}{4}\left( \frac{ne}{j} \right)^{3} \left(\frac{15ej}{s} \right)^{c\sqrt{\omega}/6 } &=& \frac{(15e^{2})^{3}}{4}\left(\frac{n}{s} \right)^{3}\left( \frac{15ej}{s} \right)^{c\sqrt{\omega}/6 - 3} \\
&\leq&  \frac{(15e^{2})^{3}}{4}\left(\frac{\sqrt{\omega}}{c} \right)^{3}\left( \frac{15e}{45} \right)^{c\sqrt{\omega}/6-3}:=h(n).
\end{eqnarray*}
Now, $h(n) = \exp\left\{O(\log \omega)-\Omega(\sqrt{\omega}) \right\} = \exp\left\{-\Omega(\sqrt{\omega}) \right\} =o(1)$, and hence for $n$ large enough, $h(n) < 1/3 < 1/e$ and so we have
\begin{eqnarray*}
\sum_{j=\lfloor \log n \rfloor + 1}^{\lfloor s/45 \rfloor}\left[\frac{1}{4}\left( \frac{ne}{j} \right)^{3}\left(\frac{15ej}{s}\right)^{c\sqrt{\omega}/6}\right]^{j} &\leq& \sum_{j=\lfloor \log n \rfloor + 1}^{\infty}(h(n))^{j} = O\left((h(n))^{\log n} \right) \\
&=& O((1/3)^{\log n}) = o(1/n).
\end{eqnarray*}
Thus, we conclude that for $s \geq cn / \sqrt{\omega}$, the sum~(\ref{eq:the_sum_to_bound}) is $o(1/n)$, uniformly in $s$. This completes the proof of part i).

\medskip

For ii), let $H$ be an induced subgraph on $s \geq cn / \sqrt{\omega}$ vertices with $\delta(H) \geq |V(H)|p/3$. By part i), we may assume that $H$ does not have a component with $\frac{s}{45}$ or fewer vertices. If $H$ has more than one component of size greater than $s/45$, then we find a pair of disjoint sets of $\left\lceil s/45 \right\rceil$ vertices each which induce no edges between them in $\Gnp$. (The assumption that $H$ is induced is necessary here.) The probability of finding such sets is at most 
\begin{eqnarray*}
    \binom{n}{\lceil s/45 \rceil}^{2}(1-p)^{\lceil s/45 \rceil^{2}} &\leq& O(\omega)\left(\frac{45ne}{s} \right)^{2s/45}e^{-s^{2}\omega/45^{2}n}\\
    &\leq& O(\omega) \left( \frac{45e\sqrt{\omega}}{c} \right)^{2s/45}e^{-cs\sqrt{\omega}/45^{2}}\\
    &=&\exp\left\{-\frac{cs\sqrt{\omega}}{45^{2}}\left(1 - O\left( \frac{\log \omega}{\sqrt{\omega}} \right) \right) \right\}\\
    &\leq& \exp\left\{ -\frac{cs\sqrt{\omega}}{50^{2}} \right\},
\end{eqnarray*}
with the final inequality holding for $n$ large enough. Thus, the probability that there exists an induced subgraph $H$ on $s \geq cn / \sqrt{\omega}$ vertices with multiple components of size greater than $s/45$ is at most 
$$
\sum_{s = \lceil cn/\sqrt{\omega} \rceil}^{n}\exp\left\{-\frac{cs\sqrt{\omega}}{50^{2}} \right\} \leq ne^{-\Omega(n)} = o(1).
$$
We conclude that a.a.s., every induced subgraph $H$ of $\Gnp$ on at least $cn / \sqrt{\omega}$ vertices with $\delta(H) \geq |V(H)|p/3$ is connected. This finishes part ii) of the proof and so the proof of the lemma is finished.
\end{proof}

\subsection{Sprinkling and P\'{o}sa Rotations}\label{long_paths}

The main result in this section is the following. 

\begin{theorem}\label{cores_hamiltonian}
Let $p = \omega / n$ and $G = \Gnp$. Then, a.a.s., for all set-pairings $(S, \mathcal{P})$ with $|S| \geq 510n / \sqrt{\omega}$, there is a nonempty sub-pairing $(S', \mathcal{P}') \subseteq (S, \mathcal{P})$ such that $G[S']$ has a Hamilton path.
\end{theorem} 

Before we prove Theorem~\ref{cores_hamiltonian}, let us show that Theorem~\ref{main_theorem} follows from it. 

\begin{proof}[Proof of Theorem~\ref{main_theorem}.]
Let $p = \omega / n$. Let $\phi: [n] \to \{1,2,\dots, c\}$ be a colouring of the vertices of $\Gnp$ with $c \le n - 510n / \sqrt{\omega}$ colour classes. We construct a set-pairing $(S(\phi), \mathcal{P}(\phi))$ associated to $\phi$ as follows. For each $j \in \{ 1, 2, \dots, c \}$, if $|\phi^{-1}(j)|$ is odd, let $v$ be the smallest vertex in $\phi^{-1}(j)$ and set $S_{j}:=\phi^{-1}(j) \setminus \{v\}$; otherwise, let $S_{j}:=\phi^{-1}(j)$. 

For each $j$ such that $s_{j}:=S_{j} > 0$, let $v_{j_{1}},v_{j_{2}}, \dots, v_{j_{s_{j}}}$ be the vertices of $S_{j}$ in increasing order, and define the pairing $\mathcal{P}_{j} := \{\{v_{j_{1}}, v_{j_{2}}\}, \{v_{j_{3}}, v_{j_{4}}\}, \dots, \{v_{j_{s_{j}-1}}, v_{j_{s_{j}}}\}\}$. Finally, define the set-pairing $(S(\phi), \mathcal{P}(\phi))$ by 
$$
S(\phi) := \bigcup_{j=1}^{c} S_{j} \quad \text{ and } \quad \mathcal{P}(\phi):= \bigcup_{j=1}^{c} \mathcal{P}_{j}.
$$
Note that 
$$
|S(\phi)| = \left| \bigcup_{j=1}^{c}S_{j} \right| =\sum_{j=1}^{c} |S_j| \geq \sum_{j=1}^{c}(|\phi^{-1}(j)|-1) = n - c \geq \frac{510n}{\sqrt{\omega}}.
$$

If $\phi$ is a linear colouring of $\Gnp$, then $(S(\phi), \mathcal{P}(\phi))$ cannot contain any nonempty sub-pairing $(S'(\phi), \mathcal{P}'(\phi))$ with a Hamilton path. But, by Theorem~\ref{cores_hamiltonian} and the fact that $|S(\phi)| \geq 510n / \sqrt{\omega}$, a.a.s.\ $(S(\phi), \mathcal{P}(\phi))$ contains such a sub-pairing, regardless which colouring $\phi$ with at most $n - 510n / \sqrt{\omega}$ colour classes is considered. Thus, we conclude that a.a.s.\ no linear colouring of $\Gnp$ with at most $n - 510n / \sqrt{\omega}$ colour classes exists, and hence a.a.s.
$$
\chi_{\text{lin}}(\Gnp) > n - \frac{510n}{\sqrt{\omega}},
$$
which finishes the proof of Theorem~\ref{main_theorem}.
\end{proof}

It remains to prove Theorem~\ref{cores_hamiltonian}. To that end, we will use the rotation-extension technique of P\'{osa}~\cite{posa1976hamiltonian}. This procedure requires a two-round exposure of the edges of $\Gnp$. That is, to generate the random graph for a given $p$, we choose two values $0 \leq p_{1}, p_{2} \leq p$ such that $p = p_{1} + p_{2} - p_{1}p_{2}$, then generate independent random graphs $\mathcal{G}(n, p_{1})$ and $\mathcal{G}(n, p_{2})$. It is easy to see that the graph obtained by taking the union of $\mathcal{G}(n,p_{1})$ and $\mathcal{G}(n,p_{2})$, and collapsing any double edges into single edges is distributed as $\Gnp$. In our case, $p = \omega/n$, and we can take $p_{1} = \frac{\omega}{2n}$ and $p_{2} = \frac{\omega}{2n} + \epsilon \ge \frac{\omega}{2n}$, where $\epsilon = O((\omega/n)^{2})$.

Rather than giving a full explanation of the technique here, we refer instead to the treatment in~\cite[Chapter 6]{frieze2016introduction}. The crucial lemma is the following, which is a straightforward consequence of~\cite[Corollary 2.10]{krivelevich2010hamiltonicity}:

\begin{lemma}\label{posa_lemma} 
Let $r$ be a positive integer, and let $G = (V,E)$ be a connected graph in which every subset $X \subseteq V$ of size $|X| \leq r$ satisfies $|N(X) \setminus X| > 2|X|.$ Suppose that the longest path in $G$ has $h \leq |V|-2$ edges. Then there are at least $r^{2}/2$ non-edges of $G$ such that the addition of any one of them results in a graph $G'$ whose longest path has at least $h+1$ edges.
\end{lemma}

In the light of the above lemma, we will call a graph $H$ \textit{good} if $H$ is connected and satisfies $|N(X) \setminus X| > 2|X|$ for every $X \subseteq V(H)$ with $|X| \leq |V(H)| / 45$. 

\medskip

We will use the following observation.

\begin{lemma}\label{sprinkling_lemma}
Let $G_{1}$ be any simple graph on vertex set $[n]$. Sample $\mathcal{G}(n,p_{2})$ and consider $G = G_{1} \cup \mathcal{G}(n,p_{2})$, collapsing double edges if needed. 
Then, the following property holds a.a.s.: for all subsets $S \subseteq[n]$ with $|S| \geq 255n / \sqrt{\omega}$ such that $G_{1}[S]$ is good, $G[S]$ contains a Hamilton path. 
\end{lemma}
\begin{proof}
Consider a set $S \subseteq [n]$ of size $|S| \geq 255n / \sqrt{\omega}$ that induces a good subgraph $H = G_{1}[S]$. Using Lemma~\ref{posa_lemma}, we will greedily build a Hamilton path on the vertices in $S$ as the edges of $\mathcal{G}(n,p_{2})$ are exposed one-by-one. For a graph $G$, define $\lambda(G)$ to be the number of edges in a longest path in $G$. 

Let $\{e_{1},e_{2},\dots,e_{r}\}$ be the edges in $\mathcal{G}(n,p_{2})$ which join pairs of vertices in $S$, listed in a random order. Note that $r$ is distributed as $\text{Bin}\left(\binom{|S|}{2}, p_{2} \right)$, which has mean asymptotic to $|S|^{2}p_{2}/2 \geq 255^{2}n / 4.$ Using Chernoff's bound (\ref{chern}), it is easy to show that $r \geq |S|^{2}p_{2} / 4$ with probability $1 - o(2^{-n})$. We condition on this outcome, and henceforth assume $r \geq |S|^{2}p_{2}/4$. Note that we only exposed the number of edges in $\mathcal{G}(n,p_{2})$ that fall into the set $S$; the locations of these edges are still unexposed.

For $1 \leq j \leq r$, inductively define $H_{j} = H_{j-1} \cup \{e_{j}\}$, where we take $H_{0} = H$. Since we assume $H$ is good, and adding edges to a good graph preserves the property of being good, $H_{j}$ is good for all $j$.

Now, fix $j \geq 0$ and condition on the outcome of $H_{j}$ (there is no conditioning necessary for $j=0$, when we simply have $H_{0} = H$). Suppose that $\lambda(H_{j}) < |S|-1$, that is, $H_{j}$ does not have a Hamilton path. By Lemma~\ref{posa_lemma}, there exists a set $B_{j}$ of at least $\frac{|S|^{2}}{2\cdot(45)^{2}} =  \frac{|S|^{2}}{4050}$ non-edges of $H_{j}$ such that if $e_{j+1} \in B_{j}$, then we have $\lambda(H_{j+1}) \geq \lambda(H_{j}) +  1$. The edge $e_{j+1}$ is uniformly distributed over pairs of vertices in $S$ which are not in the set $\{e_{1},e_{2},\dots,e_{j}\}$. Crudely, there are at most $\binom{|S|}{2}$ such pairs. Since none of the pairs in $B_{j}$ are in $\{e_{1},e_{2},\dots,e_{j}\}$ by definition, we therefore have
$$
\prob(e_{j+1} \in B_{j}\, | \, e_{1},e_{2},\dots,e_{j}) \geq \binom{|S|}{2}^{-1} \frac{|S|^{2}}{4050} \geq \frac{1}{2025}\,,
$$
given that $\lambda(H_{j}) < |S|-1$. Clearly, if $\lambda(H_{j}) = |S|-1$, then $\lambda(H_{j+1}) = |S|-1$ as well. So for any $0 \leq j \leq r-1$, either $H_{j}$ has a Hamilton path, or the length of a longest path increases by at least $1$ from $H_{j}$ to $H_{j+1}$ with probability at least $1/2025$, independently of the history up to time $j$. Thus, for as long as  $H_{j}$ has no Hamilton path, $\lambda(H_{j})$ stochastically dominates a Binomial random variable with mean $j/2025.$ In particular,
\begin{eqnarray}
\prob(H_{r} \text{ has no Hamilton path}) &\leq& \prob(\lambda(H_{r}) < |S|-1) \nonumber \\
&\leq& \prob\left(\text{Bin}\left(r, \frac{1}{2025}\right) < |S|-1\right). \label{binomial_bound}
\end{eqnarray}
Conditioned on $r \geq |S|^{2}p_{2}/4$, the $\text{Bin}(r, \frac{1}{2025})$ random variable has mean
$$
\frac{r}{2025} \geq \frac{1}{2025}\frac{|S|^{2}p_{2}}{4} \geq \frac{1}{16200}\left(\frac{255n}{\sqrt{\omega}}\right)^{2}\frac{\omega}{n} > 4n,
$$
where in the last inequality we use the fact that $255 > 2\sqrt{16200} \approx 254.558.$ By Chernoff's bound (\ref{chern}) with $t = \frac{r}{2025} - n > \frac{3}{4}\cdot\frac{r}{2025}$, 
\begin{eqnarray*}
\prob\left(\text{Bin}\left(r, \frac{1}{2025}\right) < |S|-1\right) &\leq& \prob\left(\text{Bin}\left(r, \frac{1}{2025}\right) < n\right)\\
&\leq& \exp\left\{-\frac{9}{32}\cdot\frac{r}{2025} \right\}\\
&\leq& \exp\left\{-\frac{9n}{8} \right\}\\
&=& o(2^{-n}),
\end{eqnarray*}
and so $\prob(H_{r} \text{ has no Hamilton path}) = o(2^{-n})$ as well by~(\ref{binomial_bound}).

\medskip

In summary, we have shown that, conditioned on $r \geq |S|^{2}p_{2} / 4$, the subgraph $H_{r} = G[S]$ contains a Hamilton path with probability $1 - o(2^{-n})$. Since $r \geq |S|^{2}p_{2} / 4$ also with probability $1-o(2^{-n}),$ it follows that $\prob(G[S] \text{ has no Hamilton path}) = o(2^{-n}).$ A union bound over the at most $2^{n}$ choices for the set $S$ completes the proof of the lemma.
\end{proof}

Now, we can finish the proof of Theorem~\ref{cores_hamiltonian}.
 
\begin{proof}[Proof of Theorem~\ref{cores_hamiltonian}]
By Lemma~\ref{k_core}, applied to $\mathcal{G}(n,p_{1})$ with $p_1 = p/2 = (\omega / 2)/n$, a.a.s., for every set pairing $(S, \mathcal{P})$ with 
$$
|S| \geq \frac{(255\sqrt{2})n}{\sqrt{\omega/2}} = \frac{510n}{\sqrt{\omega}},
$$ 
there exists a set-pairing $(S', \mathcal{P}') \subseteq (S, \mathcal{P})$ such that $|S'| \geq |S|/2 \ge 255n/\sqrt{\omega}$ and the subgraph $G[S']$ has minimum degree at least $|S|p_1 / 3  \ge |S'|p_1 / 3$.  
By Lemma~\ref{expansion_connectivity}, applied again to $\mathcal{G}(n,p_{1})$, a.a.s.\ every induced subgraph $H$ of $\mathcal{G}(n,p_{1})$ on at least 
$$
\frac{(255/\sqrt{2})n}{\sqrt{\omega/2}} = \frac{255n}{\sqrt{\omega}}
$$
vertices with $\delta(H) \geq |V(H)|p_{1} / 3$ is good.
	
Combining the two above observations together, we establish that a.a.s.\ in $\mathcal{G}(n,p_{1})$, for every set-pairing $(S, \mathcal{P})$ with $|S| \geq 510n / \sqrt{\omega}$, there exists an induced subgraph that is good and has at least $255n / \sqrt{\omega}$ vertices. Then, by Lemma~\ref{sprinkling_lemma}, a.a.s.\ each of these subgraphs becomes Hamiltonian after adding the edges from $\mathcal{G}(n,p_{2})$. This finishes the proof of the theorem.
\end{proof}

\section{Sparse Case: $np \le 1$ (Proof of Theorem~\ref{sparse_linear})}\label{sec:proof_sparse}

In this section, we give some results about $\chi_{\text{lin}}(\Gnp)$ in the regime $p = c/n$, $c \leq 1$. These results are implied directly by the arguments of Perarnau and Serra from~\cite{perarnau2014tree}, where the centred chromatic number of $\Gnp$ is studied under the name of \textit{tree-depth}. The relevant result therein is the following.

\begin{theorem}[\cite{perarnau2014tree} Theorem 1.2]\label{sparse_centered}
    The following hold a.a.s.:
    $$
    \chi_{\text{cen}}(\mathcal{G}(n,c/n)) = 
    \begin{cases} 
    \Theta(\log \log n) \quad& c \in (0, 1) \\ 
    \Theta(\log n) \quad& c = 1. 
    \end{cases}
    $$
\end{theorem}

Since $\chi_{\text{lin}}(G) \leq \chi_{\text{cen}}(G)$ for any graph $G$, Theorem~\ref{sparse_centered} is already enough to conclude that $\chi_{\text{lin}}(\mathcal{G}(n,c/n))$ is $O(\log \log n)$ a.a.s.\ when $c < 1$, and $O(\log n)$ a.a.s.\ when $c = 1.$ As we will see, the techniques used in \cite{perarnau2014tree} to prove the lower bounds in Theorem~\ref{sparse_centered} apply equally well to linear colourings, thus allowing us to deduce Theorem~\ref{sparse_linear}.

\medskip

Let us make a few observations. First, for any graph $G$, since any linear colouring of $G$ is necessarily a linear colouring of every subgraph of $G$, we have
\begin{equation}\label{max_subgraph}
\chi_{\text{lin}}(G) \geq \max_{H \subseteq G}\chi_{\text{lin}}(H),
\end{equation}
where the maximum is taken over all subgraphs $H$ of $G$. Next, observe that any connected subgraph of the path on $k$ vertices $P_{k}$ is necessarily a path, and hence linear and the centred colourings are equivalent on a path: 
\begin{equation}\label{eq:path}
\chi_{\text{cen}}(P_{k}) = \chi_{\text{lin}}(P_{k}). 
\end{equation}
Finally, it is well known, and easy to show, that 
\begin{equation}\label{path_colourings}
\chi_{\text{cen}}(P_{k}) = \lfloor \log_{2} k \rfloor + 1.
\end{equation}

Together, (\ref{max_subgraph}), (\ref{eq:path}), and~(\ref{path_colourings}) imply that for any graph $G$ and any component $C$ of $G$, we have
\begin{equation}\label{diameter_bound}
\chi_{\text{lin}}(G) \geq \log_{2}(\text{diam}(C)).
\end{equation}

Observation~(\ref{diameter_bound}) is all we need to prove Theorem~\ref{sparse_linear}.

\begin{proof}[Proof of Theorem~\ref{sparse_linear}]
Recall that we only need to show the lower bounds on $\chi_{\text{lin}}(\mathcal{G}(n,c/n))$; the upper bounds are implied by Theorem~\ref{sparse_centered}. 

For $c < 1$, the diameter of the largest component in $\mathcal{G}(n,c/n)$ is typically of order $\sqrt{\log n}$, but there are components of smaller cardinality with diameter of order $\log n$ (see~\cite{luczak1998random}). We conclude that a.a.s.\ $\chi_{\text{lin}}(\mathcal{G}(n,c/n)) = \Omega(\log\log n)$ by~(\ref{diameter_bound}). 

Similarly, for $c=1$, the diameter of the largest component in $\mathcal{G}(n,1/n)$ is known to be typically of order $n^{1/3}$ (see~\cite{nachmias2008critical}) implying that a.a.s.\ $\chi_{\text{lin}}(\mathcal{G}(n,c/n)) = \Omega(\log n)$.
\end{proof}

Let us mention that~\cite{perarnau2014tree} provides weaker lower bounds for the diameter of $\mathcal{G}(n,c/n)$ when $c \le 1$ but strong enough to give the same (up to a multiplicative constant) lower bounds for $\chi_{\text{lin}}(\mathcal{G}(n,c/n))$. Moreover, for $c > 1$, a.a.s.\ $\mathcal{G}(n,c/n)$ contains a path of length $\Omega(n)$ (see, for example, \cite{ajtai1981longest}) and so a.a.s.\ $\chi_{\text{lin}}(\mathcal{G}(n,c/n)) = \Omega(\log n)$. In fact, the non-existence of linear colouring is clearly a monotonic property so the same bound is implied by the fact that a.a.s.\ $\chi_{\text{lin}}(\mathcal{G}(n,1/n)) = \Omega(\log n)$. 

\bibliography{linear_colorings_refs} 

\end{document}